\documentclass{article}

\usepackage[utf8]{inputenc}
\usepackage[english]{babel}
\usepackage{amsthm}
\usepackage{amsmath}
\usepackage{amssymb}
\usepackage{url}
\usepackage{enumerate}
\usepackage{color}

\theoremstyle{plain}
\newtheorem{thm}{Theorem}
\newtheorem{lemma}{Lemma}
\newtheorem{prop}[lemma]{Proposition}
\newtheorem{cor}[lemma]{Corollary}

\theoremstyle{definition}
\newtheorem{definition}{Definition}

\theoremstyle{remark}

\newcommand{\F}{\mathbb{F}}
\newcommand{\fp}{\F_p}

\newcommand{\Pol}{\mathop{\mathrm{Pol}}}
\renewcommand{\v}{\mathop{\mathrm{v}}}

\begin{document}
\title{On the $N$th linear complexity of automatic sequences}
\author{L\'aszl\'o M\'erai and Arne Winterhof\\
  \\ \small
  Johann Radon Institute for Computational and Applied Mathematics\\ \small
  Austrian Academy of Sciences\\ \small
  Altenbergerstr.\ 69,
  4040 Linz, Austria\\ \small
  \texttt{\{laszlo.merai,arne.winterhof\}@oeaw.ac.at}
}
\maketitle

\begin{abstract}
The $N$th linear complexity of a sequence 
is a measure of
predictability. Any unpredictable sequence must have large $N$th linear complexity. However, in this paper we show that for $q$-automatic sequences over $\F_q$
the converse is not true.  

We prove that any (not ultimately periodic) $q$-automatic sequence over~$\F_q$ has $N$th linear complexity of order of magnitude $N$. 
 For some famous sequences including the Thue--Morse and Rudin--Shapiro sequence we determine the exact values of their $N$th linear complexities. 
 These are non-trivial examples of predictable sequences with $N$th linear complexity of largest possible order of magnitude.
\end{abstract}

\section{Introduction}

Let $k\ge 2$ be an integer. A {\it $k$-automatic sequence} $(u_n)$ over an alphabet ${\cal A}$ is the output sequence of a finite automaton, where the input is the $k$-ary digital expansion of $n$.
Automatic sequences have gained much attention during the last decades. For monographs and surveys about automatic sequences we refer to \cite{al,alsh,dr,RecurenceSequences}.

For a prime power $k=q$, $q$-automatic sequences $(u_n)$ over the finite field ${\cal A}=\F_q$ of $q$ elements can be characterized  
by a result of Christol, see \cite{christol} for prime $q$ and \cite{chka} for prime power $q$ as well as \cite[Theorem~12.2.5]{alsh}:
Let $$G(t)=\sum_{n=0}^{\infty}u_nt^n$$
be the {\em generating function} of the sequence $(u_n)$ over $\F_q$. 
Then $(u_n)$ is $q$-automatic over $\F_q$ if and only if $G(t)$ is algebraic over $\F_q[t]$, that is, there is a polynomial 
{\color{black} $h(s,t)\in \F_q[s,t]\setminus\{0\}$} 
such that $h(G(t),t)=0$.
(Note that for all $m=1,2,\ldots$ a sequence is $k$-automatic if and only if it is $k^m$-automatic by \cite[Theorem~6.6.4]{alsh} and even a slightly more general version of Christol's result holds:
{\color{black} For a prime $p$ and positive integers $m$ and $r$, $(u_n)$ is $p^m$-automatic over $\F_{p^r}$ if and only if $G(t)$ is algebraic over $\F_{p^r}$.})

Diem \cite{di} defined the {\em $N$th expansion complexity} $E_N(u_n)$ of $(u_n)$ as the least total degree of a nonzero polynomial $h(s,t)\in \F_q[s,t]$ with 
$$h(G(t),t)\equiv 0 \bmod t^N$$
if the first $N$ sequence elements are not all $0$ and $E_N(u_n)=0$ otherwise. Hence, the $q$-automatic sequences over $\F_q$ are exactly the sequences over $\F_q$ with  
$$E(u_n)=\sup_{N\ge 1} E_N(u_n)<\infty.$$
Sequences $(u_n)$ with small $E(u_n)$ are predictable and not suitable in cryptography. 

For example, the {\em Thue--Morse sequence} over $\F_2$ is defined by
$$t_n=\left\{ \begin{array}{cl} t_{n/2} & \mbox{if $n$ is even},\\ t_{(n-1)/2}+1 & \mbox{if $n$ is odd},
              \end{array}\right.\quad n=1,2,\ldots
$$
with initial value $t_0=0$. Taking
$$h(s,t)=s(t+1)^2+s^2(t+1)^3+t$$
its generating function $G(t)$ satisfies $h(G(t),t)=0$
and thus
$E(t_n)\le 5$. More precisely, in the proof of Corollary~\ref{cor:pattern} below we will see that $G(t)$ is not rational and thus $h(s,t)=(t+1)^3(s+G(t))(s+G(t)+\frac{1}{t+1})$
is irreducible over $\F_2(t)[s]$ and we get
$$E(t_n)=5.$$
 
The \textit{$N$th linear complexity} $L(u_n, N)$ of a sequence $(u_n)$ over $\F_q$ is the length~$L$ of a shortest linear recurrence relation satisfied by the first $N$ elements of $(u_n)$:
$$
 u_{n+L}=c_{L-1}u_{n+L-1}+\dots +c_1u_{n+1}+c_0u_n, \quad 0\leq n\leq N-L-1,
$$
for some $c_0,\ldots,c_{L-1}\in \F_q$.  
We use the convention that $L(u_n,N)=0$ if the first $N$ elements of $(u_n)$ are all zero and $L(u_n,N)=N$ if $u_0=\dots=u_{N-2}=0\ne u_{N-1}$.
The sequence $(L(u_n,N))$ is called {\em linear complexity profile} of $(u_n)$ and  
$$L(u_n)=\sup_{N\ge 1} L(u_n,N)$$ 
is the {\em linear complexity} of $(u_n)$. It is well-known (\cite[Lemma~1]{Niederreiter88-b}) that $L(u_n)<\infty$ if and only if $(u_n)$ is ultimately periodic, that is,
its generating function is rational: $G(t)=g(t)/f(t)$ with polynomials $g(t),f(t)\in \F_q[t]$.
The $N$th linear complexity is a measure for the unpredictability of a sequence as well. A large $N$th linear complexity (up to sufficiently large $N$)
is necessary (but not sufficient) for cryptographic applications. Sequences of small linear complexity are also weak in view of Monte-Carlo methods, see \cite{do,domewi,dowi,DorferWinterhof}. 
For more background on linear complexity and related measures of pseudorandomness we refer to \cite{mewi,Niederreiter2003,towi,Winterhof2010}. 
In particular, for ultimately periodic sequences the relation between linear complexity and expansion complexity was studied in \cite{meniwi}.

In this paper, we show that any $q$-automatic sequence over $\F_q$ which is not ultimately periodic has $N$th linear complexity of (best possible) order of magnitude $N$.
Hence, we provide many examples of sequences with high $N$th linear complexity which are still predictable (since $E(u_n)$ is small). 
For example, for the Thue--Morse sequence over $\F_2$ we prove, see Theorem~\ref{thm:RudinShapiro} below,
\begin{equation}\label{tmlin}
L(t_n,N)=2\left\lfloor\frac{N+2}{4}\right\rfloor,\quad N=1,2,\ldots 
\end{equation}

In Section~\ref{sec2} we prove a bound on the $N$th linear complexity of any $q$-automatic sequence over $\F_q$ which is not ultimately periodic. 
We apply this result to several famous automatic sequences including pattern sequences and sum-of-digits sequences. 
For example, for the Thue--Morse sequence our result implies
the bound
$$\left\lceil\frac{N-1}{2}\right\rceil \le L(t_n,N)\le \left\lfloor \frac{N}{2}\right\rfloor+1.$$
By $(\ref{tmlin})$ the lower bound is attained if $N\equiv 0,1\bmod 4$ and the upper bound if $N\equiv 2,3\bmod 4$. 

In Section~\ref{sec3} we determine the exact value of the $N$th linear complexity in the special case of binary pattern sequences with the all one pattern. 
The Thue--Morse sequence and the Rudin--Shapiro sequence are the simplest examples of such sequences.

Besides a small expansion complexity, the Thue--Morse and Rudin--Shapiro sequences have another deficiency, a very large correlation measure of order~$2$, see \cite{mauduitSarkozy-TM}. 
By a result of \cite{BrandstatterWinterhof}, small correlation measures of order $k$ (up to a sufficiently large $k$) imply large $N$th linear complexities. 
The converse is not true. For example, the correlation measure of order $4$ of the {\em Jacobi-sequence} $(j_n)$ of length $pq$ with two distinct odd primes $p$ and $q$ is of order of magnitude $pq$,
see \cite{risa}, which is based on the relation $j_n+j_{n+p}+j_{n+q}+j_{n+p+q}=0$ for all {\color{black}$n$ with $1\le n<pq$ and} $\gcd(n,pq)=1$. However, its linear complexity profile is quite large, 
see \cite{brwi}. This lower bound on the linear complexity profile can be obtained using an analog of the result of \cite{BrandstatterWinterhof}, see \cite{hepawawi}, 
for a modified correlation measure with {\em bounded lags}. Examples of sequences with large correlation measure of small order with bounded lags but large linear complexity profile were not known before.
However, the results of this paper for the Thue--Morse and Rudin--Shapiro sequence and the results of \cite{mauduitSarkozy-TM} (using lags $(0,1)$) show that both sequences are such examples. 

\section{Arbitrary $q$-automatic sequences over $\F_q$}
\label{sec2}

{\color{black} Ultimately periodic sequences} (that {\color{black} is linear recurrence sequences) are automatic and} correspond to rational generating functions \cite[Lemma 1]{Niederreiter88-b}. 
If $(u_n)$ is an {\color{black} ultimately periodic sequence}, then it has finite linear complexity. However, we show now that if $(u_n)$ is {\color{black} automatic but not} ultimately periodic, 
then  the $N$th linear complexity of $(u_n)$ is of order of magnitude $N$ for all $N$. 

\begin{thm}\label{thm:general}
Let $q$ be a prime power and $(u_n)$ be a $q$-automatic sequence over $\F_q$ which is not ultimately periodic. 
Let $h(s,t)=h_0(t)+h_1(t)s+\dots + h_d(t)s^d\in\F_q[s,t]$ be a non-zero polynomial with $h(G(t),t)=0$ with no rational zero.
Put $$M=\max_{0\leq i\leq d}\{\deg h_i-i\}.$$
Then we have
\[
 \frac{\displaystyle N-M}{d}\leq L(u_n,N)\leq \frac{\displaystyle (d-1)N+M+1}{d}.
\]
\end{thm}

\begin{proof}
Since $(u_n)$ is not ultimately periodic, $G(t)\not\in \F_q(t)$ is not rational by \cite[Lemma~1]{Niederreiter88-b}.

Let $g(t)/f(t)\in\F_q(t)$ be a rational zero of $h(s,t)$ modulo $t^N$ with $\deg(f)\le L(u_n,N)$ and $\deg(g)<L(u_n,N)$. 
More precisely, put $L=L(u_n,N)$. Then we have  
\begin{equation}\label{lincomp} \sum_{\ell=0}^Lc_{\ell}u_{n+\ell}=0 \quad \mbox{for }0\le n\le N-L-1
\end{equation}
for some $c_0,\ldots,c_{L}\in \F_q$ with $c_L\neq 0$.
Take
\[
f(t)=\sum_{\ell=0}^L c_{\ell}t^{L-\ell} 
\]
and 
\[
g(t)=\sum_{m=0}^{L-1}\left(\sum_{\ell=L-m}^Lc_\ell u_{m+\ell-L}\right)t^m 
\]
and verify 
$$f(t)G(t)\equiv g(t)\bmod t^N:$$
$$f(t)G(t)\equiv \sum_{\ell=0}^L c_\ell t^{L-\ell}\sum_{k=0}^{N-1}u_k t^k
\equiv \sum_{m=0}^{N-1} \sum_{\ell=0}^L c_\ell u_{m+\ell-L} t^m\equiv g(t)\bmod t^N,$$ 
where we put $m=L-\ell+k$ and used $(\ref{lincomp})$ with $n=m-L$ in the last step.

Then
\[
 h_0(t)f^d(t)+h_1(t)g(t)f^{d-1}(t)+\dots + h_d(t)g(t)^d=K(t) t^N.
\]
Here $K(t)\neq 0$ since $h(s,t)$ has no rational zero. Comparing the degrees of both sides we get
\[
 dL+M\geq N
\]
which gives the lower bound.

The upper bound for $N=1$ is trivial. For $N\geq 2$ the result follows from the well-known bound (see for example \cite[Lemma 3]{DorferWinterhof})
\[
 L(u_n,N)\leq \max\left\{L(u_n,N-1), N-L(u_n,N-1) \right\}
\]
by induction
\begin{eqnarray*} L(u_n,N)&\le& \max\left\{\frac{(d-1)(N-1)+M+1}{d},N-\frac{N-1-M}{d}\right\}\\ &\le& \frac{(d-1)N+M+1}{d},
\end{eqnarray*}
where we also used the lower bound on $L(u_n,N-1)$ to estimate $N-L(u_n,N-1)$.
\end{proof}

\subsection*{Examples}
Now we state bounds on the $N$th linear complexity of some famous automatic sequences as corollaries of Theorem \ref{thm:general}.
In the following let $p$ be a prime.

\subsubsection*{Pattern sequences}

Let $P\in \F_p^k\setminus\{00\ldots 0\}$ be a \emph{pattern of length $k$}. Let $e_P(n)$ be the number of occurrences of $P$
in the $p$-ary representation of $n$. For example if $p=2$, then $e_{11}(7)=2$, $e_1(9)=2$, $e_{11}(9)=0$ and $e_{101}(21)=2$.

For a pattern $P$ of length $k$ define the sequence $(r_n)$ by
\begin{equation*}
 r_n\equiv e_P(n) \mod p, \quad r_n\in\fp,\quad n=0,1,\dots
\end{equation*}
The sequence $(r_n)$ over $\F_p$ satisfies the following recurrence relation
\begin{equation}\label{eq:recurence}
 r_n=
 \left\{
 \begin{array}{cl}
 r_{\lfloor n/p\rfloor}+1  & \text{if } n\equiv a \mod p^k,\\
 r_{\lfloor n/p\rfloor}  & \text{otherwise,}
 \end{array}
 \right.
 n=1,2,\dots
\end{equation}
with initial value $r_0=0$, where $a$ is the {\color{black} integer in the range} $0< a <p^k$ such that its  $p$-ary expansion {\color{black} is} the pattern $P$.

Classical examples for binary pattern sequences are the Thue--Morse sequence ($p=2$, $k=1$ and $P=1$ ($a=1$)) and the {\em Rudin--Shapiro sequence} ($p=2$, $k=2$ and $P=11$ ($a=3$)).

\begin{cor}\label{cor:pattern}
Let $p$ be a prime number and $a,k$ integers with $1\leq a<p^k$. If $(r_n)$ is the pattern sequence defined by \eqref{eq:recurence}, then
\[
  \frac{N-p^k+1}{p}\leq L(r_n,N) \leq \frac{(p-1)N+p^k}{p}.
\]
\end{cor}

\begin{proof}
By the recurrence relation \eqref{eq:recurence}  we have
\begin{equation*}
 G(t)=(1+t+\dots + t^{p-1})G(t)^p+\sum_{i=0}^\infty t^{p^k i+a}.
\end{equation*}
Multiplying with $t^{p^k}-1=(t-1)^{p^k}$ we have
\begin{equation}\label{eq:pattern-gen-1}
(t-1)^{p^k}G(t)=(t-1)^{p^k+p-1}G(t)^p-t^a,
\end{equation}
{\color{black} see also \cite[Th\'eor\`em 3]{chka}.}
Put $h(s,t)=(t-1)^{p^k+p-1}s^p-(t-1)^{p^k}s-t^a$. Then $h(s,t)$ has no rational zero in $\fp(t)$. Indeed, if $g(t)/f(t)$ is a zero (with $\gcd(g,f)=1$), that is
\[
 (t-1)^{p^k+p-1}g(t)^p-(t-1)^{p^k}g(t)f(t)^{p-1}-t^af(t)^p=0,
\]
then $(t-1)\mid f(t)$ and so $(t-1)^{p^k+p-1}\mid f(t)^p$ thus $(t-1)^{p^{k-1}+1}\mid f(t)$ which means that $t-1\mid g(t)$, a contradiction.

Finally, the result follows from Theorem \ref{thm:general}.
\end{proof}

\subsubsection*{The sum-of-digits sequence}
Let $k>1$ be an integer and $\sigma_m(n)$ be the sum of digits of $n$ in the $k$-ary representation. Then define
$s_n=\sigma_m(n)\bmod k$. Clearly $(s_n)$ satisfies the following recurrence relation
\begin{equation}\label{eq:sum-of-digits}
 s_n\equiv s_{\lfloor \frac{n}{k}\rfloor}+n \mod k,\quad n=1,2,\ldots
\end{equation}
with initial value $s_0=0$.

\begin{cor}\label{cor:sum-of-digits}
Let $p$ be a prime number and $(s_n)$ be the sum-of-digit sequence modulo $p$ defined by \eqref{eq:sum-of-digits} with $k=p$. 
Then
\[
   \frac{N-1}{p} \leq L(s_n,N) \leq \frac{(p-1)N+2}{p}.
\]
\end{cor}
For $p=2$ $(s_n)$ is the Thue--Morse sequence again and in this case Corollary \ref{cor:sum-of-digits} coincides with Corollary \ref{cor:pattern}.

\begin{proof}
As before, by the recurrence relation \eqref{eq:sum-of-digits}  we have
\begin{align}\label{eq:sum-of-digits-gen-1}
 G(t)&=(1+t+\dots + t^{p-1})G(t)^p+\sum_{i=0}^\infty t^{ip }\cdot \sum_{a=1}^{p-1}at^a \notag \\
     &=\frac{1-t^p}{1-t}G(t)^p+\frac{1}{1-t^{p}}\cdot t(1-t)^{p-2}.\notag
\end{align}
Multiplying by $(1-t)^{2}$ we get that $G(t)$ is a root of
\[
 h(s,t)=(1-t)^{p+1}s^p- (1-t)^{2}s+ t.
\]
As in the previous proof, it can be shown that $h(s,t)$ has no rational zero in $\fp(t)$, thus the result follows from Theorem \ref{thm:general}.
\end{proof}

\subsubsection*{Baum--Sweet sequence}
The Baum--Sweet sequence $(b_n)$ is a $2$-automatic sequence defined by the rule $b_0=1$ and for $n\ge 1$
$$
 b_n=
 \left\{ 
 \begin{array}{cl}
  1& \text{if the binary representation of $n$ contains no block of consecutive} \\
   &  \text{$0$'s of odd length;}\\
  0& \text{otherwise.}
 \end{array}
 \right.
$$
Equivalently we have for $n\ge 1$ of the form $n=4^km$ with $m$ not divisible by $4$
\begin{equation*}
b_n=\left\{\begin{array}{ll} 0 & \mbox{if $m$ is even},\\ b_{(m-1)/2} & \mbox{if $m$ is odd}.\end{array}\right.
\end{equation*}

\begin{cor}
 Let $(b_n)$ be the Baum--Sweet sequence. Then
 \[
  \frac{N}{3}\leq L(b_n,N)\leq \frac{2N+1}{3}.
 \]
\end{cor}
\begin{proof}
It can be easily checked that the generating function $G(t)$ of $(b_n)$ is the unique root of $h(s,t)=s^3+t\cdot s+1$, {\color{black} see \cite[p.\ 403]{chka}}. 
Then the bounds follow from Theorem~\ref{thm:general}.
\end{proof}

\subsubsection*{Regular paperfolding sequence}
The value of any given term $v_n\in \F_2$ in the regular paperfolding sequence can be defined as follows. If $n = m\cdot2^k$ where $m$ is odd, then
\begin{equation*}
 v_n=
 \left\{ 
 \begin{array}{cl}
  1& \text{if } m\equiv 1 \mod 4,\\
  0& \text{if } m\equiv 3 \mod 4,
 \end{array}
 \right. n=1,2,\ldots
\end{equation*}
and any $v_0\in \F_2$.

\begin{cor}
 Let $(v_n)$ be the regular paperfolding sequence. Then
 \[
  \frac{N-3}{2}\leq L(v_n,N)\leq \frac{N}{2}+2.
 \]
\end{cor}
\begin{proof}
It can be checked that the generating function $G(t)$ of $(v_n)$ is a zero of $h(s,t)=(t^4+1)\cdot s^2+(t^4+1)\cdot s+t=(t^4+1)(s+G)(s+G+1)$
which has no rational zero. 
Then the bounds follow from Theorem \ref{thm:general}.
\end{proof}

\subsubsection*{A sequence with perfect lattice profile and perfect linear complexity profile}
The generating function $G(t)$ of the sequence $(w_n)$ over $\F_2$ defined by 
$$w_{2n}=1\quad \mbox{and} \quad w_{2n+1}=w_n+1,\quad n=0,1,\ldots$$
satisfies the functional equation
$$t(t+1)G(t)^2+(t+1)G(t)+1=0.$$
{\color{black} (This sequence or its complement are sometimes called {\it Toeplitz sequence} or {\it period doubling sequence}, see https://oeis.org/A096268.)}
Hence, $L(w_n,N)=\lfloor \frac{N+1}{2}\rfloor$. Such a linear complexity profile is called perfect. 
This is the only sequence with both a perfect linear complexity profile and a perfect 'lattice profile', see \cite{domewi} for more details. 
Sequences with the first are characterized by $w_0=1$ and $w_{2n+2}=w_{2n+1}+w_n$ but the choice of $w_{2n+1}$ is free for $n\ge 1$, see \cite{Niederreiter88-b}. 
Sequences with the latter are characterized by $w_{2n+1}=w_n+1$ but the choice of any $w_{2n}$ is free, see \cite{domewi}.

\section{Pattern sequences with the all one pattern}
\label{sec3}

In this section we slightly improve Theorem \ref{thm:general}  for the binary pattern sequences when the pattern $P$ is the all one string $11\dots 1$ of length $k$ (that is $a=2^k-1$).

\begin{thm}\label{thm:RudinShapiro}
The $N$th linear complexity $L(r_n,N)$ of the sequence $(r_n)$ defined by \eqref{eq:recurence} with $a=2^k-1$ (and $p=2$) satisfies
\[ 
 L(r_n,N)=
 \left\{
 \begin{array}{cl}
  2(2^k-1)\left\lfloor \frac{N}{4(2^k-1)}\right\rfloor+2^k & \text{if }2^k \leq N \bmod 4(2^k-1) \leq 3(2^k-1),\\[5pt]
  2(2^k-1)\left\lfloor \frac{N+2^k-2}{4(2^k-1)}\right\rfloor & \text{otherwise.}  
 \end{array}
 \right.
\]
\end{thm}

The proof of the theorem is based on the theory of continued fractions. Thus first we summarize some basic facts about them (Section \ref{sec:Cont-fract}), 
then we prove the result (Section \ref{sec:R-S}).

\subsection{Linear complexity and continued fractions}\label{sec:Cont-fract}

Now we describe the connection between linear complexity profile of a binary sequence and the continued fraction expansion of its the generating function (see for example \cite{Niederreiter88-b}).

Let $\F_2((x^{-1}))$ be the ring of formal Laurent series
\[
R=\sum_{i=-m}^\infty r_ix^{-i}, \quad r_{-m}, r_{-m+1}, \dots \in \F_2
\]
over $\F_2$ in $x^{-1}$. In fact $\F_2((x^{-1}))$ is a field and the coefficients of the inverse of a non-zero $R$ can be computed recursively, namely, if $r_{-m}=1$, then 
\[
 R^{-1}=\sum_{i=m}^\infty z_ix^{-i}
\]
with coefficients

\begin{equation}\label{S-inverse}
 z_m=1 \quad \text{and } z_{i+m}=r_{i-m}+\sum_{j=-m+1}^{i-m-1}r_{j}z_{i-j}, \quad i=1,2,\dots
\end{equation}

Every formal non-rational Laurent series $R\in\F_2((x^{-1}))$ has a unique continued fraction expansion
\[
 R=A_0+ \cfrac{1}{A_1+
	      \cfrac{1}{A_2+
	      \cfrac{1}{\ddots}
	      }},
\]
where $A_j\in\F_2[x]$ are polynomials for $j\geq 0$ and $\deg A_j\geq 1$ for $j\geq 1$.

For 
\[
R=\sum_{i=-m}^\infty r_ix^{-i}\in\F_2((x^{-1}))
\]
we define its polynomial part by
\[
\Pol(R)= \sum_{i=-m}^0 r_ix^{-i}.
\]
The polynomials $A_j$ ($j\geq 0$) are obtained recursively by 
\begin{equation*}\renewcommand{\arraystretch}{1.5}
 \begin{array}{ll}
  A_0=\Pol(R), & B_0=R-\Pol(R),\\
  A_{j+1}=\Pol(B_j^{-1}), & B_{j+1}=B_{j}^{-1}-\Pol(B_j^{-1}), \quad j\geq 0.
 \end{array}
\end{equation*}
If the continued fraction expansion is broken off after the term $A_j$ ($j\geq 0$), we get the rational convergent $P_j/Q_j$. The polynomials $P_j, Q_j$ can be calculated recursively by

\begin{equation}\label{eq:Q-recursion}
\begin{array}{llll}
  P_{-1}=1,  &   P_0=A_0, &   P_j=A_jP_{j-1}+P_{j-2}, &  j\geq 1, \\[5pt]
    Q_{-1}=0,  &   Q_0=1,   &   Q_j=A_jQ_{j-1}+Q_{j-2}, &  j\geq 1. 
\end{array}
\end{equation}

The following formulas are shown by straightforward induction on $j$:
\begin{align} 
 & \deg (Q_j)=\sum_{h=1}^j \deg A_h \quad \text{for } j\geq 1, \label{eq:deg-Q} \\[10pt]
  & P_{j-1}Q_j+P_jQ_{j-1}=1 \quad \text{for } j\geq 1,\label{eq:P-Q-co-prime} \\[12pt]
  & R=\frac{P_j+B_jP_{j-1}}{Q_j+B_jQ_{j-1}} \quad \text{for } j\geq 0. \label{eq:S-P-Q}
\end{align}

With $x=t^{-1}\in\F_2(t)$ we obtain that the generating function of the sequence $(r_n)$ over $\F_2$ is
\[
G=\sum_{n=0}^\infty r_nx^{-n}.
\]

The following lemma \cite[Theorem 1]{Niederreiter88-b} gives an explicit description of the linear complexity profile of the sequence $(r_n)$ in terms of the polynomials $Q_j$ that are obtained 
from the continued fraction expansion of 
\begin{equation}\label{eq:def-R}
 R=x^{-1}G=\sum_{i=1}^{\infty}r_{i-1}x^{-i}.
\end{equation}

\begin{lemma}\label{lemma:Niederreiter}
For any $N\geq 1$ the $N$th linear complexity $L(r_n,N)$ is given by
\[
 L(r_n,N)=\deg Q_j
\]
where $j\geq 0$ is uniquely determined by the condition
\[
 \deg Q_{j-1}+\deg Q_{j}\leq N < \deg Q_{j}+\deg Q_{j+1}.
\]
\end{lemma}

We define the (exponential) valuation $\v$ on $\F_2((x^{-1}))$ as
\[
 \v(R)=m, \quad \text{if } R=\sum_{i=-m}^\infty r_ix^{-i}\in\F_2((x^{-1})) \text{ and } r_{-m}\neq 0.
\]
For $R=0$ we put $\v(R)=-\infty$. We have the following properties of $\v$. For $R,S\in \F_2((x^{-1})) $ we have
\begin{enumerate}[(i)]
 \item $\v(RS)=\v(R)+\v(S)$,
 \item $\v(R+S)\leq \max\{\v(R),\v(S)\}$,
 \item $\v(R+S)= \max\{\v(R),\v(S)\}$ if $\v(R)\neq \v(S)$.
\end{enumerate}

The valuation $\v$ extends the degree function on $\F_2[x]$: for $f\in\F_2[x]$ we have $\v(f)=\deg f$ and
\[
 \v\left( \frac{f}{g}\right)=\deg f -\deg g \quad \text{for } f,g\in\F_2[x], g\neq 0.
\]

\subsection{Proof of Theorem \ref{thm:RudinShapiro}}\label{sec:R-S}

It follows from the functional equation \eqref{eq:pattern-gen-1} for the generating function of the sequence $(r_n)$ that the function $R$ defined in \eqref{eq:def-R} satisfies 
\begin{equation}\label{eq:function-equation}
 (1+x)R^2+R+U^{2^k}x^{-2^k}=0,
\end{equation}
where $U=\sum_{i=0}^\infty x^{-i}$.

\begin{lemma}\label{lemma:Q-mod-x}
 If for a polynomial $Q\in\F_2[x]$ we have
 \[
  QU^{2^{k}}=\sum_{i=-\deg Q}^{\infty} b_ix^{-i},
 \]
 then  $b_i=b_j$ for $i,j\geq 0$ with $i\equiv j \mod 2^{k}$ and
\[
 Q\equiv b_1 x^{2^{k}-1} + b_2 x^{2^{k}-2}+ \dots + b_{2^{k}} \mod x^{2^{k}}+1.
\]

\end{lemma}

\begin{proof}
 Write $Q=x^d+q_{d-1}x^{d-1}+\dots +q_1x+q_0$. Then for $i \geq 0$ the coefficient of $QU^{2^k}$ of $x^{-i}$ is
 \begin{equation*}
  b_{i}=\sum_{j: \ i+j\equiv 0 \bmod 2^{k}} q_j,
 \end{equation*}
and the result follows. 
\end{proof}

\begin{proof}[Proof of Theorem \ref{thm:RudinShapiro}]
We prove the following assertions by induction.
 \begin{enumerate}[(i)]
  \item \label{lemma:R-S-1-ii} $A_1=x^2+x+1$ if $k=1$ and $A_1=x^{2^k}+x$ if $k\geq 2$.
  \item \label{lemma:R-S-1-iii}For $j\geq 2$, $A_j=x^2+1$ if $k=1$ and for $j\geq 1$,
  $A_{2j}=x^{2^k-2}+x^{2^k-4}+\dots +x^2+1$, $A_{2j+1}=x^{2^k}+1$ if $k\geq 2$. 
     \item \label{lemma:R-S-1-iv}For $j\geq 0$, $Q_{j}\equiv 1 \mod x+1$ if $k=1$ and  $Q_{2j}\equiv 1 \mod x^{2^{k-1}}+1$, $Q_{2j+1}\equiv x+1 \mod x^{2^{k-1}}+1$ for $k\geq 2$. 
 \end{enumerate} 
 
Then the result follows from Lemma \ref{lemma:Niederreiter},  \eqref{lemma:R-S-1-ii}, \eqref{lemma:R-S-1-iii} and (\ref{eq:deg-Q}).
 
The first part follows from straightforward computation. Namely, observe that the first $2^{k+1}$ elements of the sequence $(r_n)$ are zeros, except the elements $r_n$ with $n=2^k-1$ and $n=2^{k+1}-2$. Thus
\[
  R=\frac{1}{x^{2^k}}+\frac{1}{x^{2^{k+1}-1}}+ \sum_{i=2^{k+1}+1}^\infty r_{i-1}x^{-i},
\]
whence by \eqref{S-inverse} we have
\[
 R^{-1}=x^{2^k}+x+x^{2-2^k}+\dots
\]
so
\[
 A_1=\Pol (R^{-1})=
 \left\{
 \begin{array}{ll}
  x^{2}+x+1 & \text{if $k=1$},\\
  x^{2^k}+x & \text{if $k\geq 2$}
 \end{array}
\right.
\]
which proves \eqref{lemma:R-S-1-ii}.

By Corollary \ref{cor:pattern}, $R$ is irrational and thus $\deg A_j \geq 1$ for all $j\geq 1$. Now by \eqref{eq:deg-Q}, \eqref{eq:P-Q-co-prime}, (\ref{eq:S-P-Q}) and the properties of $\v$ it follows for $l\geq 1$ that
 \begin{equation}\label{eq:v-Q-0}
  \v(Q_{l-1}R-P_{l-1})=-\v(Q_{l})
 \end{equation}
whence
\begin{equation}\label{eq:v-Q}
  \v(Q_{l-1}^2R^2-P^2_{l-1})=-2\v(Q_{l}).
\end{equation}

Consider
\begin{align}
T&=(x+1)P_{l-1}^2+P_{l-1}Q_{l-1}+U^{2^k}x^{-2^k}Q_{l-1}^2   \label{eq:T-def-1} \\ 
 &=(x+1)\left(Q^2_{l-1}R^2-P_{l-1}^2 \right)+Q_{l-1}(P_{l-1}+Q_{l-1}R),\label{eq:T-def-2}
\end{align}
where the second equality follows from  \eqref{eq:function-equation}.
Now we have
\begin{equation}\label{eq:A-T-deg}
 \v(T)=-\v(A_{l}), \quad l\geq 1,
\end{equation}
by \eqref{eq:deg-Q}, \eqref{eq:v-Q-0},   \eqref{eq:v-Q}, \eqref{eq:T-def-2}  and the properties of $\v$.
It follows from \eqref{eq:Q-recursion}, \eqref{eq:P-Q-co-prime} and \eqref{eq:S-P-Q}, that
\begin{equation}\label{eq-A_l-inverse}
 Q_{l-1}(P_{l-1}+Q_{l-1}R)=Q_{l-1}\frac{1}{A_{l}Q_{l-1}+Q_{l-2}+B_{l}Q_{l-1}}=\frac{1}{A_{l}+\frac{Q_{l-2}}{Q_{l-1}}+B_{l}}
\end{equation}
where we have
\begin{equation}\label{eq:deg-B-Q_{l-2}/Q_{l-1}}
 \v(B_{l}), \v\left(\frac{Q_{l-2}}{Q_{l-1}}\right)<0.
\end{equation} 
On the other hand 
\begin{align}\label{eq:large-deg}
 \v\left((x+1)\left(Q^2_{l-1}R^2-P_{l-1}^2 \right)\right)&=1-2\v(Q_l)=-2\v(A_l)+1-2\v(Q_{l-1}) \notag \\
 &\leq -2\v(A_l)+1-2\v(A_1)<-2\v(A_l)
\end{align}
for $l\ge 2$.

We now prove \eqref{lemma:R-S-1-iii} and \eqref{lemma:R-S-1-iv} by induction. We remark that the assertion \eqref{lemma:R-S-1-iv} for $Q_{0}$ and $Q_1$ follows from \eqref{eq:Q-recursion} and \eqref{lemma:R-S-1-ii}. 

Take $l\geq 2$. If $k=1$ or $k\geq 2$ and $l=2j+1$ is odd, then $Q_{l-1}\equiv 1 \mod x^{2^{k-1}}+1$, thus $Q_{l-1}^2\equiv 1 \mod x^{2^{k}}+1$. By Lemma \ref{lemma:Q-mod-x}, the coefficients of $x^{-i}$ ($i=1,\dots, 2^k-1$) in $Q_{l-1}^2U^{2^k}$ are 
all zero but the coefficient of $x^{-2^k}$ is one. Since $R$ is irrational, the degree $\deg A_l$ and so $\v(T)$ cannot be zero. Thus by \eqref{eq:T-def-1}  $x^{-2^k}$ is the leading term of $T$ so $\deg (A_l)=2^k$ by \eqref{eq:A-T-deg}. 

By Lemma \ref{lemma:Q-mod-x}, the coefficient of $x^{-i}$ ($i=2^k+1,\dots, 2^{k+1}-1$) in $Q_{l-1}^2U^{2^k}$ is zero. Thus by \eqref{eq:large-deg}  
\[
 Q_{l-1}(P_{l-1}+Q_{l-1}R)= x^{-2^k}+x^{2^{k+1}}+\sum_{i=2^{k+1}}^{\infty}a_ix^{-i}.
\]
It follows from \eqref{eq:deg-B-Q_{l-2}/Q_{l-1}} that
\[
 A^{-1}_{l}= x^{-2^k}+x^{2^{k+1}}+\sum_{i=2^{k+1}}^{\infty}a'_ix^{-i},
\]
whence $A_l=x^{2^k}+1$ by \eqref{S-inverse}.

If $k\geq 2$ and $l=2j$ is even, the proof of \eqref{lemma:R-S-1-iii} is similar. Since $Q_{l-1}\equiv x+1 \mod x^{2^{k-1}}+1$, thus $Q_{l-1}^2\equiv x^2+1 \mod x^{2^{k}}+1$. By Lemma \ref{lemma:Q-mod-x}, the coefficients of $x^{-i}$ ($i=1,\dots, 2^k-3$) in $Q_{l-1}^2U^{2^k}$ are all zero but the coefficient of $x^{-2^k+2}$ is one. Thus by \eqref{eq:T-def-1}  $x^{-2^k+2}$ is the leading term of $T$ so $\deg (A_l)=2^k-2$ by \eqref{eq:A-T-deg}. By Lemma \ref{lemma:Q-mod-x}, the coefficient of $x^{-i}$ ($i=2^k-1,\dots, 2^{k+1}-1$) in $Q_{l-1}^2U^{2^k}$ are all zeros except for  $i=2^{k+1}-2$.
Thus By \eqref{eq:large-deg}  
\[
 Q_{l-1}(P_{l-1}+Q_{l-1}R)= x^{-2^k+2}+x^{2^{k}} +x^{-2^{k+1}-2}+\sum_{i=2^{k+1}-1}^{\infty}a_ix^{-i}.
\]
It follows from \eqref{eq:deg-B-Q_{l-2}/Q_{l-1}} that
\[
 A^{-1}_{l}= x^{-2^k+2}+x^{2^{k}} +x^{-2^{k+1}-2}+\sum_{i=2^{k+1}-1}^{\infty}a'_ix^{-i},
\]
whence $A_{l}=x^{2^k-2}+x^{2^k-4}+\dots +x^2+1$ by \eqref{S-inverse}, which proves \eqref{lemma:R-S-1-iii}.

Since $x^{2^k}+1\mid A_l$ for $l\geq 2$, we have by \eqref{eq:Q-recursion} that
\[
Q_l=A_lQ_{l-1}+Q_{l-2}\equiv Q_{l-2} \mod x+1,
\]
which proves \eqref{lemma:R-S-1-iv}.
\end{proof}

\subsection*{Acknowledgement}
The authors are partially supported by the Austrian Science Fund FWF Project 5511-N26 which is part of the Special Research Program {\color{black}``}Quasi-Monte Carlo Methods: Theory and Applications''.

\end{document}